\documentclass[leqno,a4paper,12pt]{article}
\usepackage{setspace} 
\title{Additive average {S}chwarz with adaptive coarse spaces: scalable algorithms for multiscale problems 
}
\author{Leszek Marcinkowski\footnotemark[2]
        \and Talal Rahman\footnotemark[3]}

\usepackage{srcltx}

\usepackage{graphicx}
\usepackage{epstopdf}
\usepackage{labelfig}

\usepackage{latexsym}

\usepackage{slashbox} 

\usepackage{amsmath}
\usepackage{amssymb}

\usepackage{amsfonts}

\usepackage{lineno} 
\usepackage{hyperref}

\newenvironment{keywords}{\textbf{Keywords:}}{\\}
\newenvironment{AMS}{\textbf{AMS:}}{\\}
\newenvironment{proof}{\textbf{Proof:}}{$\Box$}
\newtheorem{remark}{Remark}
\newtheorem{theorem}{Theorem}
\newtheorem{lemma}{Lemma}

\date{}

\begin{document}
\maketitle

\renewcommand{\thefootnote}{\fnsymbol{footnote}}

\footnotetext[2]{Faculty of Mathematics, Univeristy of Warsaw, Banacha 2, 02-097 Warszawa, Poland ({\tt Leszek.Marcinkowski@mimuw.edu.pl})}
\footnotetext[3]{Department of Computing, Mathematics and Physics, Western Norway University of Applied Sciences, Inndals\-veien 28, 5063 Bergen, Norway ({\tt Talal.Rahman@hvl.no})}

\begin{abstract}
We present an analysis of the additive average Schwarz preconditioner with two newly proposed adaptively enriched coarse spaces which was presented at the 23rd International conference on domain decomposition methods in Korea, for solving second order elliptic problems with highly varying and discontinuous coefficients. It is shown that the condition number of the preconditioned system is bounded independently of the variations and the jumps in the coefficient, and depends linearly on the mesh parameter ratio H/h, that is the ratio between the subdomain size and the mesh size, thereby retaining the same optimality and scalablity of the original additive average Schwarz preconditioner.
\end{abstract}




\section{Introduction}

Additive Schwarz methods are considered among the most effective preconditioners for solving algebraic systems arising from the discretization of elliptic partial differential equations, because they generate algorithms that are easy to implement, inherently parallel, scalable and fast. With proper enrichment of the coarse spaces, the methodology has recently been quite successfully applied to multiscale problems with highly heterogeneous and varying coefficients, a class of problems which most standard iterative solvers have difficulty to solve efficiently. Additive average Schwarz is one of the simplest of all additive Schwarz preconditioners because it is easy to construct and quite straightforward to analyze. Unlike most additive Schwarz preconditioners, its local subspaces are defined on non-overlapping subdomains, and no explicit coarse grid is required as the coarse space is simply defined as the range of an averaging operator. By enriching its coarse space with functions corresponding to the bad eigen modes of the local stiffness it has been shown numerically in a recent presentation, cf. \cite{Marcinkowski:2017:MCSAAS}, that the method can be made both scalable and robust with respect to any variation and jump in the coefficient when solving multiscale problems. The purpose of this paper is to give a complete analysis of the method presented in the paper.                     

The additive average Schwarz method in its original form, was first introduced for the second order elliptic problems in \cite{Bjorstad:1997:ASM}, where the method was applied to and analyzed for problems with constant coefficient in each subdomain and jumps only across subdomain boundaries. The method was further extended to non-matching grids using the mortar dicretization in \cite{Bjorstad:2003:ASM,Rahman:2005:ASM,Rahman:2011:SPCG}, and to fourth order problems in \cite{Feng:2002:AAS}.
For multiscale problems where the coefficient may be highly varying and discontinuous also inside subdomains, the method has been analyzed in \cite{Dryja:2010:AAS}, where it has been shown that the condition number of the preconditioned system depends linearly on the jump of coefficient in the subdomains' layers, and quadratically on the ratio of the coarse to the fine mesh parameters. The method has very recently been extended to the Crouzeix-Raviart finite volume discretization, cf. \cite{Loneland:2016:AAS,Loneland:2016:AASM}, showing similar results. All these results on the additive average Schwarz method however suggest that the method by itself can not be robust for multiscale problems unless some form of enrichment of the coarse space is made, which eventually led the research to the approach newly presented in the 23rd international conference on domain decomposition, cf. \cite{Marcinkowski:2017:MCSAAS}, where adaptively chosen eigenfunctions of certain local eigenvalue problems, extended by zero to the rest of the domain, are added to the standard average Schwarz coarse space.  
This idea of enriching the coarse space with eigenfunctions for improved convergence goes back several years, e.g. the paper \cite{Bjorstad-Koster-Krzyzanowski:HPC:2002,Bjorstad-Krzyzanowski:PPAM2001} on a substructuring domain decomposition method and \cite{Chartier:2003:SAMGe} on an algebraic multigrid method. Using the idea to solve multiscale problems started only very recently, with the papers \cite{Galvis:2010:DDM,Galvis:2010:DDMR,Nataf:2011:CSC}. Since then, a number of other works have emerged proposing algorithms based on solving different eigenvalue problems, see e.g. \cite{chung2016adaptive,Dolean:2012:ATL,Efendiev:2012:RTD,Efendiev:2012:RDD,atle2015harmonic,Spilane:2014:ARC} for those using the additive Schwarz framework for their algorithms, and \cite{kim2015bddc,KRR:2015:FDMACS,klawonn2016comparison,mandel2007adaptive,sousedik2013adaptive,Calvo:2016:ACP,kim2016bddc,klawonn:2016:adaptive,oh:2016:BDDC,pechstein2016unified} for those using the FETI-DP or the BDDC framework for their algorithms.
 
Throughout this paper, we use the following notations: $x\lesssim y$ 
and $w\gtrsim z$ denote that there exist positive constants $c$ and $C$ independent of mesh parameters and the jump of coefficients such that $x\leq c y$ and $w\geq C z$, respectively.

The remainder of the paper is organized as follows: we state the discrete problem in Section~\ref{sec:problem}, the additive Schwarz method in Section~\ref{sec:schwarz-met}, and the coarse spaces in Section~\ref{sec:coarse-space}. In Section~\ref{sec:cond-bound} the condition number bound is given and proved. Some numerical results are then presented in Section~\ref{sec:num_exp}.

\section{Discrete problems}\label{sec:problem}
We consider a model multiscale elliptic problem on a polygonal domain $\Omega$ in the 2D.  
We seek $u^*\in H^1_0(\Omega)$ such that
\begin{equation}\label{eq:discrete-pr}
 a(u^*,v)=\int_\Omega f v \: d x \qquad \forall v \in H^1_0(\Omega),
\end{equation}
where 
\begin{equation*}
 a(u,v)=\int_\Omega \alpha(x) \nabla u \nabla v \: d x
\end{equation*}
$f\in L^2(\Omega)$, and $\alpha \in L^\infty(\Omega)$ is the positive coefficient function.
We assume that there exists an $\alpha_0>0$ such that $\alpha(x) \geq \alpha_0$ in $\Omega$. Since we can scale the problem by $\alpha_0^{-1}$, we can further assume that $\alpha_0=1$.

We introduce the triangulation $T_h(\Omega)=T_h=\{\tau\}$ consisting of the triangles $\tau$, and assume that this triangulation is quasi-uniform in the sense of \cite{Braess:2007:FE,Brenner:2008:MTF}. Let $V^h$ be the discrete finite element space consisting of continuous piecewise linear functions with zero on the boundary $\partial\Omega$, i.e. 
\begin{equation}
 V^h=\{v\in C(\overline{\Omega}): v_{|\tau} \in P_1(\tau) \ \ \forall \tau \in T_h, \ \ v=0 \ \ \partial\Omega\},
\end{equation}
where $P_1(\tau)$ is the space of linear polynomials over the triangle $\tau\in T_h$.
Each $v\in V^h$ can be represented in the standard nodal basis as $v=\sum_{x\in N_h} v(x)\phi_x$, where $N_h$ is set of the nodal points, i.e. all vertices of triangles in $T_h$ which are not on $\partial\Omega$. 

The corresponding discrete problem is then to find $u^*_h \in V^h$ such that 
\begin{equation}
 a(u_h^*,v)=\int_\Omega f v \: d x \qquad \forall v \in V^h.
\end{equation}
The Lax-Milgram theorem yields that this problem has a unique solution.


Note that for $u,v\in V^h$ their gradients are constant over each triangle $\tau\in T_h$. Thus we see that 
$\int_\tau \alpha(x)\nabla u \nabla v \: dx = (\nabla u_{|\tau} ) (\nabla v_{|\tau} )\int_\tau \alpha(x)\: dx$.
Hence without any loss of generality we can assume that  $\alpha$ is piecewise constant over the triangles of the triangulation $T_h$.

Further we assume that we have a coarse partition  of $\Omega$ into open connected polygonal subdomains $\{\Omega_k\}_{k=1}^N$ such that 
$\overline{\Omega}=\bigcup_{k=1}^N \overline{\Omega}_k$, where each 
$\overline{\Omega}_k$ is a sum of some closed triangles of $T_h$. 
Let $H=\mathrm{diam}(\Omega_k)$ be the coarse mesh parameter, and $\Gamma$ the interface defined as $\Gamma=\bigcup_{k=1}^N \partial\Omega_k \setminus \partial\Omega$.


\begin{figure}[htb]
\centerline{ 
\SetLabels
\L\B (.45*.53)  \boldmath{${\Omega_k}$} \\
\L\B (.60*.01)  \boldmath{${\Omega_{k}^{\delta}}$} \\
\endSetLabels
\AffixLabels{
\includegraphics[width=0.32\textheight]{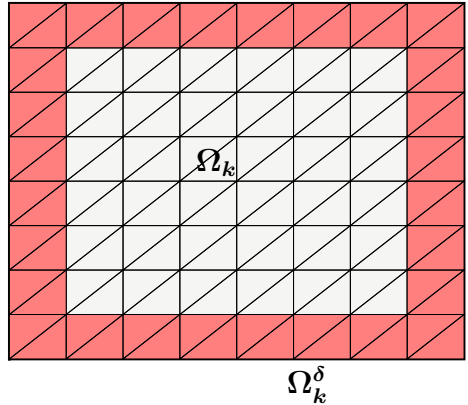}
} 
} 
\caption{$\Omega_{k}^{\delta}$ is the layer corresponding to the subdomain $\Omega_k$, and consisting of elements (triangles) of $T_h(\Omega_k)$ touching the subdomain boundary $\partial\Omega_k$.}
\label{fig:layer}
\end{figure}

Each subdomain inherits its triangulation $T_h(\Omega_k)=\{\tau\in T_h: \tau \subset \Omega_k\}$ from the $T_h(\Omega)$. Consequently, we define the local finite element space $V^h(\Omega_k)$ as the space of functions of $V^h$, restricted to $\overline{\Omega}_k$, and 
\begin{equation} \label{eq:loc_space_zero_bnd}
 V^h_0(\Omega_k)=V^h(\Omega_k)\cap H^1_0(\Omega_k).
\end{equation}
Let $\Omega_k^\delta\subset \Omega_k, k=1,\ldots,N$, be the open discrete layers, where each $\Omega_k^\delta$ is defined as the interior of the sum of all closed triangles $\tau \in T_h(\Omega_k)$
such that $\partial\tau\cap \partial\Omega_k\not= \emptyset$, cf. Figure~\ref{fig:layer}.
We introduce the local maximums and minimums of coefficients over a subdomain and its layer, as
\begin{equation}\label{eq:coeff-min-max}
\begin{array}{lcrlcr}
 \underline{\alpha}_k&:=&\min_{x\in \Omega_k}\alpha,&
  \qquad  \overline{\alpha}_k&:=&\max_{x\in \Omega_k}\alpha, \\
  \underline{\alpha}_{k,\delta}&:=&\min_{x\in \Omega_k^\delta}\alpha,& \qquad 
    \overline{\alpha}_{k,\delta}&:=&\max_{x\in \Omega_k^\delta}\alpha. 
\end{array}
\end{equation}
Let also $\Omega_h,\partial\Omega_h,\Omega_{k,h},\partial \Omega_{k,h}$, and $\Gamma_h$ denote
the sets of nodal points which are on $\Omega,\partial\Omega,\Omega_k,\partial \Omega_k$, and $\Gamma$, respectively.

\section{Additive Schwarz method}\label{sec:schwarz-met}
The method (cf. \cite{Marcinkowski:2017:MCSAAS}) is constructed based on the abstract scheme of the additive Schwarz method (ASM), cf. e.g. \cite{Smith:1996:DDP,Toselli:2005:DDM}. Accordingly, for each local subproblem, $V_k \subset V^h$ the subspace corresponding to the subdomain $\Omega_k$, is defined as the space of functions of $V_0^h(\Omega_k)$ extended by zero to the rest of $\Omega$, i.e.
\begin{equation*}
 V_k=\{u\in V^h: u(x)=0 \quad x\not \in \Omega_k\}, \quad k=1,\ldots,N.
\end{equation*}
For the coarse problem we propose two different coarse spaces for the Schwarz method, $V_0^{\scalebox{0.5}{TYPE}}\subset V^h$ where $\scalebox{0.8}{TYPE}$ is either $\scalebox{0.8}{LAYER}$ or $\scalebox{0.8}{SUBD}$, defined later in Section~\ref{sec:coarse-space}, cf. (\ref{eq:coarse_space}).
The corresponding projections $P_k:V^h\rightarrow V_k$, are defined as
\begin{equation}
 a(P_ku,v)=a(u,v) \qquad \forall v\in V_k \qquad k=1,\ldots,N
\end{equation}
and $P_0^{\scalebox{0.5}{TYPE}}:V^h\rightarrow V_0^{\scalebox{0.5}{TYPE}}$, as
\begin{equation}
 a(P_0^{\scalebox{0.5}{TYPE}}u,v)=a(u,v) \qquad \forall v\in V_0^{\scalebox{0.5}{TYPE}}, \qquad \scalebox{0.8}{TYPE}\in\{\scalebox{0.8}{LAYER,SUBD}\}. 
\end{equation}
Now following the Schwarz scheme, and writing the additive Schwarz operator $P^{\scalebox{0.5}{TYPE}}:V^h\rightarrow V^h$ as
\begin{equation}\label{eq:ASM-oper}
 P^{\scalebox{0.5}{TYPE}}u=P_0^{\scalebox{0.5}{TYPE}}u + \sum_{k=1}^N P_k u \qquad  \scalebox{0.8}{TYPE}\in\{\scalebox{0.8}{LAYER,SUBD}\}.
\end{equation}
we can replace the original problem (\ref{eq:discrete-pr}) with the following problem, cf. e.g. \cite{Smith:1996:DDP,Toselli:2005:DDM}:
\begin{equation}\label{eq:ASMprob}
  P^{\scalebox{0.5}{TYPE}}u^*_h=g^{\scalebox{0.5}{TYPE}} \qquad  \scalebox{0.8}{TYPE}\in\{\scalebox{0.8}{LAYER,SUBD}\}.
\end{equation}
where $g^{\scalebox{0.5}{TYPE}}=g_0^{\scalebox{0.5}{TYPE}}u+\sum_k g_k$ with $g_0^{\scalebox{0.5}{TYPE}}=P_0^{\scalebox{0.5}{TYPE}}u_h^*$ and $g_k=P_ku_h^*$.
The function on the right hand side of (\ref{eq:ASMprob}) can be computed without knowing $u^*_h$, cf. \cite{Smith:1996:DDP,Toselli:2005:DDM}. The condition number bounds for $ P^{\scalebox{0.5}{TYPE}}$ are given in Theorem~\ref{thm:cond-bound} in Section~\ref{sec:cond-bound}.

\section{Coarse spaces and interpolation operators}\label{sec:coarse-space}
We start by introducing the two coarse spaces (cf. \cite{Marcinkowski:2017:MCSAAS}) for the additive average Schwarz method, they correspond to $\scalebox{1.0}{TYPE}=\scalebox{1.0}{SUBD}$ and $\scalebox{1.0}{TYPE}=\scalebox{1.0}{LAYER}$. In the base, both have the same classical additive average Schwarz coarse space which is defined as the range of the average interpolation operator $I_0: V^h\rightarrow V^h$ (cf. \cite{Bjorstad:1997:ASM,Dryja:2010:AAS,Loneland:2016:AASM}):
\begin{equation}\label{eq:aver-interp-op}
 I_0 u(x)= \left\{ 
 \begin{array}{ll}
 u(x) &\qquad x\in \Gamma_h,\\
 \overline{u}_k &\qquad x\in \Omega_{k,h},
 \end{array}   
 \qquad k=1,\ldots,N,
 \right.
\end{equation}
where $\overline{u}_k=\frac{1}{n_k}\sum_{x\in \partial\Omega_{k,h}} u(x)$ with $n_k$ being the number of nodal points in $\partial\Omega_{k,h}$, in other words, the discrete average of $u$ over the boundary of the subdomain. This space is then enriched with functions that are adaptively selected eigenfunctions of a specially constructed generalized eigenvalue problem, cf. (\ref{eq:eigen-prl}), defined locally in each subdomain, and extended by zero to the rest of the domain. This local generalized eigenvalue problem is of either $\scalebox{1.0}{TYPE}=\scalebox{1.0}{SUBD}$ or $\scalebox{1.0}{TYPE}=\scalebox{1.0}{LAYER}$, differing in the bilinear form $b_k^{\scalebox{0.5}{TYPE}}(\cdot,\cdot)$ used in (\ref{eq:eigen-prl}). The resulting coarse space in either case is then the classical additive average Schwarz coarse space $I_0V^h$ enriched with the corresponding functions, cf. (\ref{eq:eigen-prl})--(\ref{eq:coarse_space}) below.

Here is how the generalized eigenvalue problem is defined locally in each subdomain $\Omega_k$. Find all eigen pairs:
$(\lambda_j^{k,\scalebox{0.5}{TYPE}},\psi_j^{k,\scalebox{0.5}{TYPE}})\in (\mathbb{R},V^h_0(\Omega_k))
$ such that 
\begin{equation} \label{eq:eigen-prl}
\begin{array}{lcl}
   a_k(\psi_j^{k,\scalebox{0.5}{TYPE}},v)=
   \lambda_j^{k,\scalebox{0.5}{TYPE}}
        b_k^{\scalebox{0.5}{TYPE}}(\psi_j^{k,\scalebox{0.5}{TYPE}},v)
        &\qquad &\forall v \in V^h_0(\Omega_k),\\
 b_k^{\scalebox{0.5}{TYPE}}(\psi_j^{k,\scalebox{0.5}{TYPE}},\psi_j^{k,\scalebox{0.5}{TYPE}})=1 &&
\end{array} 
\end{equation}
where $a_k(\cdot,\cdot)$ and $b_k^{\scalebox{0.5}{TYPE}}(\cdot,\cdot)$ for $\scalebox{1.0}{TYPE}\in\{\scalebox{1.0}{SUBD},\scalebox{1.0}{LAYER}\}$, are symmetric bilinear forms defined as follows,
\begin{eqnarray*}
  a_k(u,v)&:=&\int_{\Omega_k}\alpha\nabla u\nabla v\: d x
\end{eqnarray*}
and
\begin{eqnarray*}
   b_k^{\scalebox{0.5}{SUBD}}(u,v)
     &:=&\int_{\Omega_k}\underline{\alpha}_k\nabla u\nabla v\: d x,\\
  b_k^{\scalebox{0.5}{LAYER}}(u,v)
     &:=&\int_{\Omega_k^\delta}\underline{\alpha}_{k,\delta}\nabla u\nabla v\: d x     +
     \int_{\Omega_k\setminus \Omega_k^\delta}\alpha\nabla u\nabla v\: d x.
\end{eqnarray*}
Note that if two eigenvalues are different then their respective eigenspaces and eigenfunctions are both $a_k(\cdot,\cdot)$- and $ b_k^{\scalebox{0.5}{TYPE}}(\cdot,\cdot)$ orthogonal to each other. In case of an eigenvalue of multiplicity larger than one, we consider all its eigenfunctions as one.
We also order the eigenvalues in the decreasing order as
$ \lambda_1^{k,\scalebox{0.5}{TYPE}}\geq \lambda_2^{k,\scalebox{0.5}{TYPE}}\geq \ldots,
 \lambda_{N_k}^{k,\scalebox{0.5}{TYPE}}>0$ where $N_k$ is the dimension of $V^h_0(\Omega_k)$.
 
\begin{remark}
We see that  $1 \leq \lambda_j^{k,\scalebox{0.5}{SUBD}} \leq \frac{\overline{\alpha}_k}{\underline{\alpha}_k}$ and  $1 \leq \lambda_j^{k,\scalebox{0.5}{LAYER}} \leq \frac{\overline{\alpha}_{k,\delta}}{\underline{\alpha}_{k,\delta}}$. 
Thus when $\alpha$ is constant in $\Omega_k$ all eigenvalues of both eigenvalue problems are equal to one, and when $\alpha$ is constant in the layer $\Omega_k^\delta$ all eigenvalues
$\lambda_j^{k,\scalebox{0.5}{LAYER}}$ are equal to one.
\end{remark} 
  
For further use, we extend $\psi_j^{k,\scalebox{0.5}{TYPE}}$ by zero to the rest of the domain $\Omega$, denoting the extended function by the same symbol. Now let
\begin{equation} \label{eq:enrichment-space}
 W_k^{\scalebox{0.5}{TYPE}}:=\mathrm{Span}
 (\psi_j^{k,\scalebox{0.5}{TYPE}})_{k=1}^{M_k^{\scalebox{0.5}{TYPE}}},
 \qquad  \scalebox{0.8}{TYPE}\in\{\scalebox{0.8}{LAYER,SUBD}\},
 \end{equation}
 where $0 \leq M_k^{\scalebox{0.5}{TYPE}} < N_k$ is a number either preset by the user or chosen adaptively (it is the number of eigenvalues whose magnitudes are smaller than or equal to a given threshold). We assume that if an eigenvalue which has been selected to be included has multiplicity larger than one, then all its eigenfunctions will be included in the $W_k^{\scalebox{0.5}{TYPE}}$. Consequently, $\lambda_{M_k^{\scalebox{0.5}{TYPE}}+1}^{\scalebox{0.5}{TYPE}}<\lambda_{M_k^{\scalebox{0.5}{TYPE}}}^{\scalebox{0.5}{TYPE}}$.
 $M_k^{\scalebox{0.5}{TYPE}}=0$ means enrichment is not required in the subdomain $\Omega_k$.
 Our coarse spaces are then defined as follows,
\begin{equation} \label{eq:coarse_space}
 V_0^{\scalebox{0.5}{TYPE}}=I_0V^h + \sum_{k=1}^N W_k^{\scalebox{0.5}{TYPE}},
 \qquad  \scalebox{0.8}{TYPE}\in\{\scalebox{0.8}{LAYER,SUBD}\}.
\end{equation}

The two operators, which we need for the analysis, are defined here. The first one being a $b_k^{\scalebox{0.5}{TYPE}}(\cdot,\cdot)$ orthogonal projection operator 
$
 \Pi_k^{\scalebox{0.5}{TYPE}}:V^h_0(\Omega_k)\rightarrow
   V^h_0(\Omega_k),
$
defined as
\begin{equation}\label{eq:spectral-proj}
\Pi_k^{\scalebox{0.5}{TYPE}} v= \sum_{j=1}^{M_k^{\scalebox{0.5}{TYPE}}} b_k^{\scalebox{0.5}{TYPE}}
(v,\psi_j^{k,\scalebox{0.5}{TYPE}}) \psi_j^{k,\scalebox{0.5}{TYPE}},
\quad  \scalebox{0.8}{TYPE}\in\{\scalebox{0.8}{LAYER,SUBD}\},
\end{equation}
where $ (\psi_j^{k,\scalebox{0.5}{TYPE}})_j$ form the $b_k^{\scalebox{0.5}{TYPE}}(\cdot,\cdot)$-orthonormal eigen-basis of $V^h_0(\Omega_k)$, cf. (\ref{eq:eigen-prl}). The second operator is defined in the following paragraph.

We note that for any $u\in V^h$, the function $w=u-I_0u \in V^h$ equals zero on $\partial\Omega_k $, for each $k$, thus $\Pi_k^{\scalebox{0.5}{TYPE}} (u-I_0u)_{|\Omega_k}$, 
$\scalebox{0.8}{TYPE}\in\{\scalebox{0.8}{LAYER,SUBD}\}$ is properly defined, cf. (\ref{eq:spectral-proj}). It is then further extended by zero to the rest of the domain obtaining a function in  $W_k^{\scalebox{0.5}{TYPE}}$, cf. (\ref{eq:enrichment-space}). Further the extended  $\Pi_k^{\scalebox{0.5}{TYPE}} (u-I_0u)_{|\Omega_k}$ will be denoted by the same symbol.
Then, $I_0^{\scalebox{0.5}{TYPE}}:V^h\rightarrow V_0^{\scalebox{0.5}{TYPE}}$ is defined as follows,
\begin{equation}\label{eq:coarse-enr-oper}
 I_0^{\scalebox{0.5}{TYPE}}=I_0 u + \sum_{k=1}^N \Pi_k^{\scalebox{0.5}{TYPE}} (u-I_0u)
 \qquad   
  \scalebox{0.8}{TYPE}\in\{\scalebox{0.8}{LAYER,SUBD}\}.
\end{equation}

\section{Condition number bound}\label{sec:cond-bound}
We present our main theoretical result here, cf. Theorem~\ref{thm:cond-bound}, which gives an upperbound for the condition number of the preconditioned system (\ref{eq:ASMprob}).
\begin{theorem}\label{thm:cond-bound}
Let $ P^{\scalebox{0.5}{TYPE}}$ be the additive Schwarz operator, where $\scalebox{0.8}{TYPE}\in\{\scalebox{0.8}{LAYER,SUBD}\}$, as defined in (\ref{eq:ASM-oper}). Then it holds that
\begin{equation*}
 \left( \min_k \frac{1}{\lambda_{M_k^{\scalebox{0.5}{TYPE}}+1}^{\scalebox{0.5}{TYPE}}}\right) \frac{H}{h} a(u,u)\lesssim a(P^{\scalebox{0.5}{TYPE}}u,u) \lesssim  a(u,u)
\quad \forall u \in V^h, \; \scalebox{0.8}{TYPE}\in\{\scalebox{0.8}{LAYER,SUBD}\},
\end{equation*} 
where $H=\mathrm{diam}(\Omega_k)$ and $\lambda_{M_k^{\scalebox{0.5}{TYPE}}+1}^{\scalebox{0.5}{TYPE}}$ is the $M_k^{{\scalebox{0.5}{TYPE}}+1}$-th eigenvalue of (\ref{eq:eigen-prl}) (cf. also  (\ref{eq:enrichment-space})).
\end{theorem}
\begin{proof}
The proof follows from the theory of abstract Schwarz framework which requires three key assumptions to be verified, cf. e.g.  \cite{Smith:1996:DDP,Toselli:2005:DDM}.
The first assumption is on the existence of a stable splitting (cf. e.g. Assumption~2.2 in \cite{Toselli:2005:DDM}) which will be proved in Lemma~\ref{prop:stable-decomp} below. It is straightforward to verify the other two assumptions. The local stability assumption is satisfied with the stability constant being equal to one since only the exact bilinear form is used for the local bilinear forms. While the assumption on the Cauchy-Schwarz relationship between the local subspaces is satisfied with the spectral radius of the matrix of constants of the strengthened Cauchy-Schwarz inequalities being equal to one, since the local subspaces are orthogonal to each other.   
\end{proof}



We prove the stability assumption needed in the proof of Theorem \ref{thm:cond-bound}, in Lemma \ref{prop:stable-decomp}. In order for that we need a few estimates which we state first.
\subsection*{A spectral estimate}
A well-known spectral estimate which is used in our analysis, is presented here.
Let $V$ be a finite dimensional space with a symmetric positive definite bilinear form $b(u,v)$, and a symmetric non-negative definite form $a(u,v)$. We say that 
$\lambda\in \mathbb{R}$ is an eigenvalue of the following eigenvalue problem, if there exists a non-zero eigenfunction $\phi_\lambda\in V$ such that
\begin{equation}\label{eq:abs-eigenpr}
 a(\phi_\lambda,v)=\lambda b(\phi_\lambda,v) \qquad \forall v \in V.
\end{equation}
Now let $V_\lambda$ be the eigen space associated with the eigenvalue $\lambda$, and let $\Pi_\mu$ for $\mu>0$ be the $b(\cdot,\cdot)$-orthogonal projection onto the space 
\begin{equation}
  W_\mu=\sum_{\lambda > \mu} V_\lambda.
\end{equation}
where the sum is taken over all eigenvalues greater than $\mu$ (a threshold).
Then, we have the following lemma.
\begin{lemma}
For $u\in V$, it holds that
\begin{equation}
 \|u - \Pi_\mu u\|_a^2\leq \mu \|u\|_b^2,
\end{equation}
where $\|u\|_b=b(u,u)$ and $\|u\|_a=a(u,u)$. 
\end{lemma}
\begin{proof}
The proof is quite standard, yet, for the completeness we give its proof here.
If $\lambda_k$ and $\lambda_l$ are two different eigenvalues then $V_{\lambda_k}$ are $V_{\lambda_l}$ are orthogonal to each other with respect to both $a(\cdot,\cdot)$ and $b(\cdot,\cdot)$.
Let $u\in V$, then since $V=\sum_{\lambda} V_\lambda$ and the eigen spaces are both $b(\cdot,\cdot)$- and $a(\cdot,\cdot)$ orthogonal, we can uniquely write $u=\sum_\lambda \psi_\lambda$ with $\psi_\lambda\in V_\lambda$.
Now using (\ref{eq:abs-eigenpr}) we have
\begin{equation*}
\|u\|_b^2=\sum_\lambda \|\psi_\lambda\|_b^2 \qquad \|u\|_a^2
=\sum_\lambda \|\psi_\lambda\|_a^2
=\sum_\lambda \lambda\|\psi_\lambda\|_b^2.
\end{equation*}
Using the above and noting that $u-\Pi_\mu u=u-\sum_{\lambda >\mu}  \psi_\lambda= \sum_{\lambda\leq\mu}  \psi_\lambda$,
we immediately get
\begin{equation*}
 \|u-\Pi_\mu u\|_a^2=\sum_{\lambda\leq \mu} \|\psi_\lambda\|_a^2
 =\sum_{\lambda\leq \mu} \lambda\|\psi_\lambda\|_b^2
 \leq \sum_{\lambda\leq \mu} \mu\|\psi_\lambda\|_b^2
 \leq \sum_{\lambda} \mu\|\psi_\lambda\|_b^2 = \mu \|u\|_b^2,
\end{equation*}
what ends the proof.
\end{proof}

Now, applying the above lemma to the operator $\Pi_k^{\scalebox{0.5}{TYPE}}$, for
 $\scalebox{0.8}{TYPE}\in\{\scalebox{0.8}{LAYER,SUBD}\}$, and taking 
 $\mu=\lambda_{M_k^{\scalebox{0.5}{TYPE}}+1}^{\scalebox{0.5}{TYPE}}$
 we get the following estimate for the operator,
\begin{equation}\label{eq:enrich-est}
 a_k(u-\Pi_k^{\scalebox{0.5}{TYPE}}u,u-\Pi_k^{\scalebox{0.5}{TYPE}}u)
 \leq  \lambda_{M_k^{\scalebox{0.5}{TYPE}}+1}^{\scalebox{0.5}{TYPE}}
  b_k^{\scalebox{0.5}{TYPE}}(u,u)
  \quad 
  \scalebox{0.8}{TYPE}\in\{\scalebox{0.8}{LAYER,SUBD}\}.
\end{equation} 

\subsection*{An estimate of the coarse interpolation operator}
What we need first is an estimate of the average interpolation operator $I_0$ in the norms induced by the two local bilinear forms $b_k^{\scalebox{0.5}{TYPE}}(\cdot,\cdot)$, for $\scalebox{0.8}{TYPE}\in\{\scalebox{0.8}{LAYER,SUBD}\}$.
\begin{lemma}\label{lem:umI0u}
Let $I_0$ be the average interpolation operator as defined in (\ref{eq:aver-interp-op}).
Then 
\begin{equation}
 b_k^{\scalebox{0.5}{TYPE}}(u-I_0u,u-I_0u)\lesssim 
  \frac{H}{h} a_k(u,u)\qquad \forall u \in V^h
  \quad 
  \scalebox{0.8}{TYPE}\in\{\scalebox{0.8}{LAYER,SUBD}\}.
\end{equation}
\end{lemma}
\begin{proof}
In case of $\scalebox{0.8}{TYPE=SUBD}$ the proof follows from \cite{Bjorstad:1997:ASM} (see also \cite{Loneland:2016:AASM}). Namely, we get
\begin{equation*}
 b_k^{\scalebox{0.5}{SUBD}}(u-I_0u,u-I_0u)=
  \underline{\alpha}_k\|\nabla (u -I_0 u)\|_{L^2(\Omega_k)}^2\lesssim \underline{\alpha}_k\frac{H}{h}\|\nabla u\|_{L^2(\Omega_k)}^2 \leq \frac{H}{h}a_k(u,u).
\end{equation*}
For the first inequality above we refer to \cite{Bjorstad:1997:ASM}, while the second inequality follows from the definition of $\underline{\alpha}_k^2$, cf. (\ref{eq:coeff-min-max}).  

In the case $\nabla (u-I_0u)=\nabla u $ on each fine triangle  $\tau \in T_h(\Omega_k)$ which is not inside the $\Omega_k^\delta$, since $I_0u$ is constant there, cf. (\ref{eq:aver-interp-op}). Hence 
\begin{eqnarray*}
 b_k^{\scalebox{0.5}{LAYER}}(u-I_0u,u-I_0u)&=&
  \|\alpha_k^{1/2}\nabla (u-I_0 u)\|_{L^2(\Omega_k\setminus \Omega_k^\delta)}^2+
   \underline{\alpha}_{k,\delta}\|\nabla (u-I_0 u)\|_{L^2(\Omega_k^\delta)}^2\\
  &=&
  \|\alpha_k^{1/2}\nabla u\|_{L^2(\Omega_k\setminus \Omega_k^\delta)}^2+  \underline{\alpha}_{k,\delta}\|\nabla u-I_0 u\|_{L^2(\Omega_k^\delta)}^2 \\
 & \leq&
   \|\alpha_k^{1/2}\nabla u\|_{L^2(\Omega_k)}^2+  \underline{\alpha}_{k,\delta}\|\nabla( u-I_0 u)\|_{L^2(\Omega_k^\delta)}^2 \\
 & =&a_k(u,u)+   \underline{\alpha}_{k,\delta}\|\nabla( u-I_0 u)\|_{L^2(\Omega_k^\delta)}^2.
\end{eqnarray*}
To estimate the second term above, we utilize a triangle inequality and (\ref{eq:coeff-min-max}), giving
\begin{eqnarray}\nonumber
   \underline{\alpha}_{k,\delta}\|\nabla( u-I_0 u)\|_{L^2(\Omega_k^\delta)}^2
  &\lesssim&  
    \underline{\alpha}_{k,\delta}\|\nabla u\|_{L^2(\Omega_k^\delta)}^2 
   +  \underline{\alpha}_{k,\delta}\|\nabla I_0 u\|_{L^2(\Omega_k^\delta)}^2\\
   &\leq&
    a_k(u,u)+  \underline{\alpha}_{k,\delta}\|\nabla I_0 u\|_{L^2(\Omega_k^\delta)}^2. 
    \label{eq:proof-avintop}
\end{eqnarray}
Again, to estimate the second term in (\ref{eq:proof-avintop}), the proof is analogous to the proof of Lemma~4.4 in \cite{Loneland:2016:AASM}, however, for the sake of completeness we provide a short proof here.

Note that for a triangle $\tau \in T_h$ we have the following equivalence,
\begin{equation*}
 |\nabla u|_{L^2(\tau)}^2\lesssim \sum_{i,j\in\{1,2,3\}}|u(x_i)-u(x_j)|^2\lesssim  |\nabla u|_{L^2(\tau)}^2,
\end{equation*}
where the sum is taken over all pairs of vertices of $\tau$.
Using this and the discrete equivalence of the $L^2$ norm over a 1D element, we get
\begin{eqnarray*}
 \|\nabla( I_0 u)\|_{L^2(\Omega_k^\delta)}^2
 &=&\sum_{\tau \subset \Omega_k^\delta} \|\nabla I_0 u\|_{L^2(\tau)}^2
\lesssim
 \sum_{x\in  \partial \Omega_{k,h}} (u(x)-\overline{u}_k)^2 \\
 &\lesssim&
h^{-1}\|u-\overline{u}_k\|_{L^2(\partial\Omega_k)}^2. 
\end{eqnarray*}
By a trace theorem, Poincar\'e inequality and a scaling argument, we get
\begin{equation*}
 h^{-1}\|u-\overline{u}_k\|_{L^2(\partial\Omega_k)}^2 \lesssim 
 \frac{H}{h}\|\nabla u\|_{L^2(\Omega_k)}^2.
\end{equation*} 
Now using the last two estimates and (\ref{eq:coeff-min-max}), the second term in the right hand side of 
(\ref{eq:proof-avintop}) can be bounded as
\begin{equation*}
  \underline{\alpha}_{k,\delta}\|\nabla( I_0 u)\|_{L^2(\Omega_k^\delta)}^2\lesssim
  \underline{\alpha}_{k,\delta} \frac{H}{h}\|\nabla u\|_{L^2(\Omega_k)}^2
 \leq  \frac{H}{h}a_k(u,u),
\end{equation*} 
what ends the proof.
\end{proof}

The following lemma gives a stability estimate for the coarse operator $I_0^{\scalebox{0.5}{TYPE}}$ (cf. (\ref{eq:coarse-enr-oper})), where $\scalebox{0.8}{TYPE}\in \{\scalebox{0.8}{SUBD},\scalebox{0.8}{LAYER}\}$.

 
\begin{lemma}\label{lemma:coarse-interp-est}
Let $I_0^{\scalebox{0.5}{TYPE}}$ be a coarse operator defined in (\ref{eq:coarse-enr-oper}). Then we have
\begin{equation}
 a(u-I_0^{\scalebox{0.5}{TYPE}}u,u - I_0^{\scalebox{0.5}{TYPE}}u)\lesssim
 \max_k  \lambda_{M_k^{\scalebox{0.5}{TYPE}}+1}^{\scalebox{0.5}{TYPE}}
 \frac{H}{h}a(u,u) \qquad   
  \scalebox{0.8}{TYPE}\in\{\scalebox{0.8}{LAYER,SUBD}\}
\end{equation} 
for any $u\in V^h$. Here $\lambda_{M_k^{\scalebox{0.5}{TYPE}}+1}^{\scalebox{0.5}{TYPE}}$
is as in Theorem \ref{thm:cond-bound}.
\end{lemma}
\begin{proof}
Define $w=u -I_0u$. Clearly, $w$ is equal to zero on the interface $\Gamma$.
Note that
\begin{equation}
u-I_0^{\scalebox{0.5}{TYPE}}u=\sum_k(I-\Pi_k^{\scalebox{0.5}{TYPE}}) w,
\end{equation}
which is also equal to zero on the interface $\Gamma$.
Then
\begin{equation}
a(u-I_0^{\scalebox{0.5}{TYPE}}u,u - I_0^{\scalebox{0.5}{TYPE}}u)=
\sum_k a_k((I-\Pi_k^{\scalebox{0.5}{TYPE}}) w,(I-\Pi_k^{\scalebox{0.5}{TYPE}}) w)
\end{equation}
Using (\ref{eq:enrich-est}) and Lemma~\ref{lem:umI0u}, we get
\begin{eqnarray*}
a_k((I-\Pi_k^{\scalebox{0.5}{TYPE}}) w,(I-\Pi_k^{\scalebox{0.5}{TYPE}}) w)
&\lesssim&  \lambda_{M_k^{\scalebox{0.5}{TYPE}}+1}^{\scalebox{0.5}{TYPE}}
b_k^{\scalebox{0.5}{TYPE}}(w,w) \\
&\lesssim &
\lambda_{M_k^{\scalebox{0.5}{TYPE}}+1}^{\scalebox{0.5}{TYPE}}\frac{H}{h}
a_k(u,u).
\end{eqnarray*}
Summing over all subdomains ends the proof.
\end{proof}

We are now ready to give a proof of the stability assumption of the splitting needed for the proof in Theorem \ref{thm:cond-bound}. 
\begin{lemma}[Stable Splitting] \label{prop:stable-decomp}
For $u\in V^h$, let $u_0=I_0^{\scalebox{0.5}{TYPE}}u\in V_0^{\scalebox{0.5}{TYPE}}$, for $\scalebox{0.8}{TYPE}\in\{\scalebox{0.8}{LAYER,SUBD}\}$ and
$u_k\in V_k$, for $k =1,\ldots,N$, equals to $(u-u_0)_{|\Omega_k}$ extended by zero to the rest of the domain.
Then  we have 
$ u=u_0+\sum_{k=1}^N u_i $,
and 
\begin{equation*}
 a(u_0,u_0)+\sum_{k=1}^Na(u_k,u_k)\lesssim 
\max_k \lambda_{M_k^{\scalebox{0.5}{TYPE}}+1}^{\scalebox{0.5}{TYPE}}\frac{H}{h}
a(u,u), \qquad   
  \scalebox{0.8}{TYPE}\in\{\scalebox{0.8}{LAYER,SUBD}\}.
\end{equation*}
\end{lemma}
\begin{proof}
It is not difficult to see that the splitting $u=u_0+\sum_{k=1}^N u_i$ is valid following the definition of $u_k$. We prove the stability of this splitting as follows. Using a triangle inequality we get 
\begin{equation}\label{eq:proof-stable-dec}
 a(u_0,u_0)\lesssim a(u,u)+a(u - I_0^{\scalebox{0.5}{TYPE}}u,u-I_0^{\scalebox{0.5}{TYPE}}u).
\end{equation}
Noting that 
\begin{equation*}
    a(u_k,u_k)=a_k(u_k,u_k)=a_k(u-I_0^{\scalebox{0.5}{TYPE}}u,u-I_0^{\scalebox{0.5}{TYPE}}u),
    \qquad k=1,\ldots,N,
\end{equation*}
we have
\begin{equation}\label{eq:proof-stable-dec:1}
\sum_{k=1}^Na(u_k,u_k)= \sum_{k=1}^Na_k(u-I_0^{\scalebox{0.5}{TYPE}}u,u-I_0^{\scalebox{0.5}{TYPE}}u) = a(u - I_0^{\scalebox{0.5}{TYPE}}u,u-I_0^{\scalebox{0.5}{TYPE}}u).
\end{equation}
Now, combining (\ref{eq:proof-stable-dec}) and (\ref{eq:proof-stable-dec:1}), and then applying Lemma~\ref{lemma:coarse-interp-est}, we get the proof.
\end{proof}

\section{Numerical experiments} \label{sec:num_exp}
We present some simple numerical experiments to validate our theory. The problem is considered on a unit square with homogeneous boundary conditions and the right-hand side function $f(x)=2\pi^2 \sin(\pi x) \sin(\pi y)$.
We chose the following distribution of the coefficient $\alpha$ for the experiment: it consists of a background, channels crossing inside subdomains and stretching out of subdomains (continuous across) or stopping at subdoman boundaries (discontinuous across), and inclusions along subdomain boundaries (placed at the corners), see Fig. \ref{fig:channelinclusion} for an illustration with $3\times 3$ subdomains. $\alpha$ is equal to $\alpha_b$ in the background, $\alpha_c$ in the crossing channels, and $\alpha_i$ in the inclusions placed at the corners. Experiments have been performed using the proposed methods as preconditioners with the conjugate gradients iteration, and stopping the iterations when the residual norm in each case is reduced by the factor 5e{-6}. The multiplicative version (cf. \cite{Rahman:2011:SPCG}) of the preconditioner has also been used, the results of which are reported for comparison. As expected, cf. e.g. \cite{Rahman:2011:SPCG}, the multiplicative version converges twice as fast as the additive version in terms of the number of iterations, its condition number being one fourth of that of the additive version.   

\begin{figure}[htb]
\centerline{
\hspace{-.8cm}\includegraphics[width=0.5\textwidth]{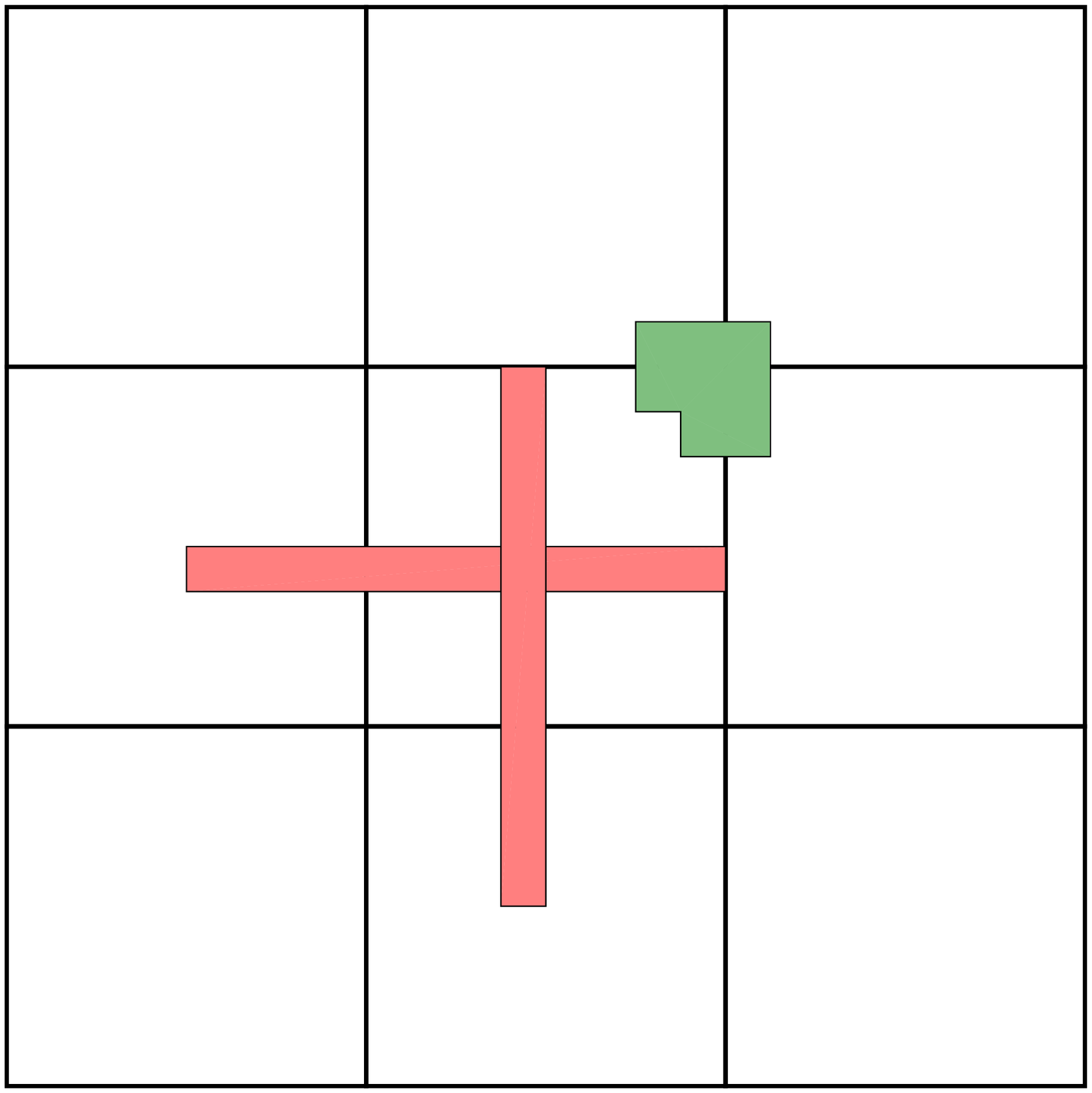}\hspace{-.3cm}
\includegraphics[width=0.5\textwidth]{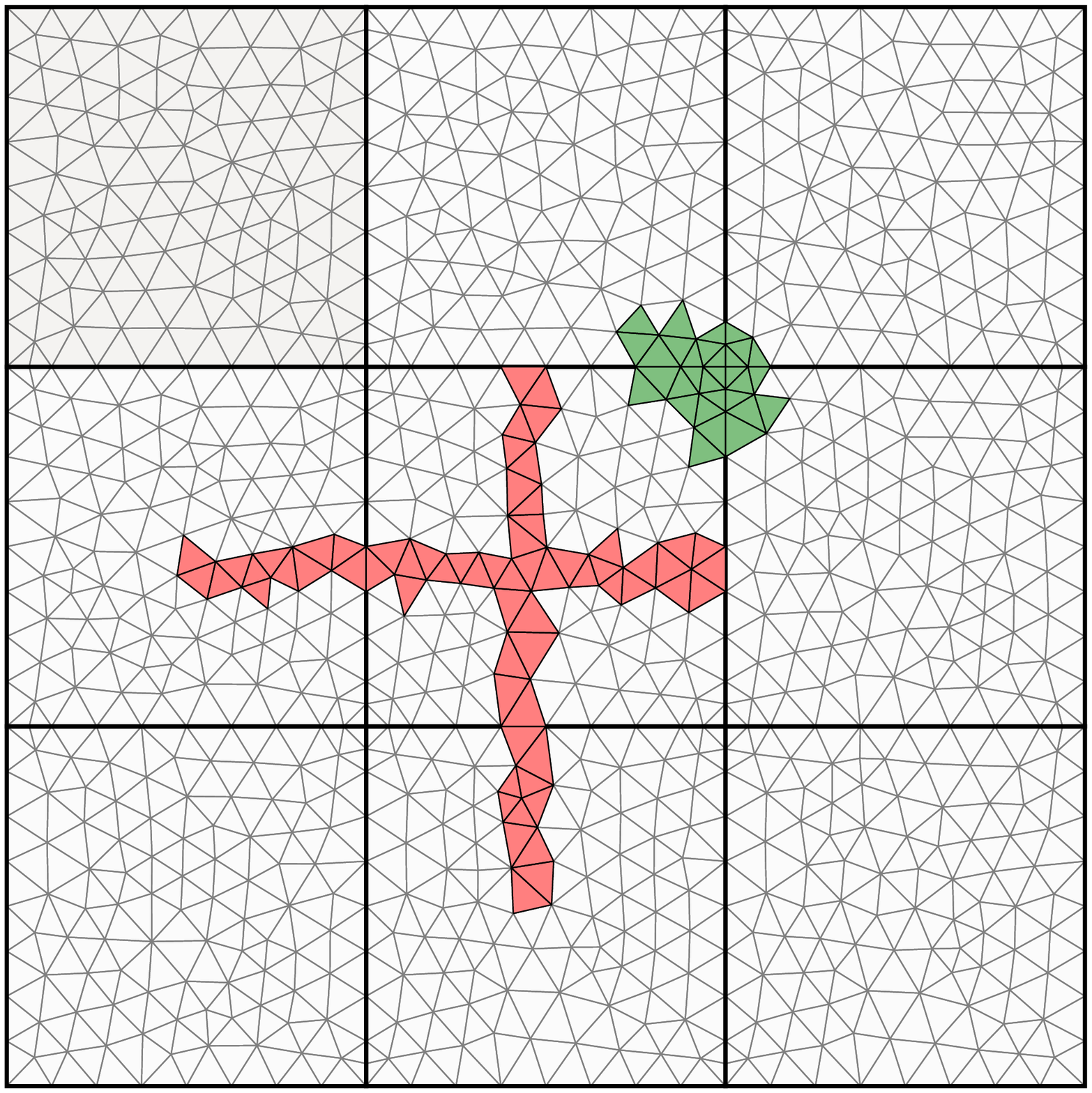}
}
\caption{On the left, a domain comprising of $3\times 3$ subdomains, with one inclusion (shaded green) placed at a subdomain corner, and a pair of channels (shaded red) crossing each other inside a subdomain, and they are both either continuous or discontinuous accross the subdomain's boundary. On the right, a finite element triangulation of the domain, showing the corresponding channels and the inclusion in the triangulation. $\alpha$ is equal to $\alpha_b$ in the background, $\alpha_c$ in the crossing channels, and $\alpha_i$ in the inclusions.}
\label{fig:channelinclusion}
\end{figure}

In our first experiment, we study the convergence behavior of the proposed method as we vary the mesh size parameters $H$ and $h$, and the jump in the coefficient $\alpha$, while for the enrichment including only those eigenfunctions for the enrichment, whose eigenvalues are greater than a given threshold (adaptive enrichment). We refer to Table \ref{tbl:1}, where the number of iterations required to converge and a condition number estimate (in parentheses) for each experiment are given for the $\scalebox{0.8}{TYPE}=\scalebox{0.8}{LAYER}$ algorithm. 

The results of $\scalebox{0.8}{TYPE}=\scalebox{0.8}{SUBD}$ algorithm have been very similar to those of $\scalebox{0.8}{TYPE}=\scalebox{0.8}{LAYER}$, however, there has been a significant difference in the number of eigenfunctions included, between the two, which we will discuss later in this section, cf. Fig. \ref{fig:bargraph}. The threshold for the eigenvalues has been set to 100.

\begin{table}[h]
\centering
\small
\begin{tabular}{|c|ccc|ccc|}
\hline 
& \multicolumn{3}{|c|}{Additive (ADD)} & \multicolumn{3}{|c|}{Multiplicative (MLT)} \\
\hline
\backslashbox{$h$}{$H$} & 1/3 & 1/6 & 1/9 & 1/3 & 1/6 & 1/9\\
\hline
  1/18 & 34 (5.84e{1}) &&                                  & 17 (1.49e1) &&\\
  1/36 & 56 (1.35e{2}) & 52 (5.71e{1}) &                 & 28 (3.40e1) & 26 (1.45e1) &\\
  1/54 & 70 (2.13e{2}) & 67 (9.32e{1}) & 58 (6.03e{1}) & 35 (5.34e1) & 34 (2.36e1) & 29 (1.53e1)\\
\hline
  1/18 & 37 (5.80e{1}) &&                                  & 19 (1.48e1) &&\\
  1/36 & 53 (1.34e{2}) & 53 (5.60e{1}) &                 & 27 (3.36e1) & 26 (1.43e1) &\\
  1/54 & 67 (2.12e{2}) & 68 (9.19e{1}) & 59 (5.94e{1}) & 33 (5.33e1) & 34 (2.32e1) & 29 (1.51e1)\\
\hline
\end{tabular}
\vspace{5mm}
\caption{
Showing number of iterations and a condition number estimates (in parentheses) for varying $H$ and $h$. The left block of results correspond to the additive version (ADD), while the right block corresponds to the multiplicative version (MLT) of the average Schwarz method. The first three rows correspond to the coefficient distribution $\alpha_b=1$, $\alpha_c=1e2$, and $\alpha_i=1e4$, and the last three rows correspond to the coefficient distribution $\alpha_b=1$, $\alpha_c=1e4$, and $\alpha_i=1e6$.}
\label{tbl:1}
\end{table}

\begin{table}[h]
\centering
\small
\begin{tabular}{|c| c c c c c c |}
\hline
 & 0 & 2 & 4 & 5 & 6 & 7 \\
\hline
 ADD & 508 (3.09e{6}) & 431 (1.08e{6}) & 186 (1.98e{4}) & 61 (4.62e2) & 49 (4.78e1) & 48 (4.71e1)\\
 MLT & 259 (7.73e{5}) & 218 (2.71e{5}) & 94 (4.96e{3}) & 30 (1.15e3) & 24 (1.22e1) & 24 (1.20e1)\\
\hline
\end{tabular}
\vspace{5mm}
\caption{
Showing number of iterations and a condition number estimates (in parentheses) for fixed number of eigenfunctions for enrichment (respectively 0, 2, 4, 5, 6, 7 of eigenfunctions in each subdomain). The first line (ADD) of results correspond to the additive version, while the second line (MLT) corresponds to the multiplicative version of the method. Here $H=1/6$, $h=1/36$, $\alpha_b=1$, $\alpha_c=1e4$, and $\alpha_i=1e6$.}
\label{tbl:2}
\end{table}

The columns of Table \ref{tbl:1} correspond to the subdomain size $H$ and the rows correspond to the mesh size $h$. In order to have the same pattern in the distribution of $\alpha$ even when we vary the size of the subdomains, that is to have crossing channels inside subdomains and inclusions at subdomain corners, as illustrated in Fig. \ref{fig:channelinclusion}. So we let the size of the channels and the inclusions to vary proportionally with $H$, which is somewhat artificial but inevitable for the purpose of this experiment. We should however mention that they do not vary with the mesh size $h$, so that each column of the table corresponds to the same set of channels and inclusions. We note that the diagonal entries of the table correspond to the same mesh size ratio $\frac{H}{h}$. The corresponding condition number estimates as seen from the table are very close to each other, suggesting that the proposed preconditioners are scalable in the sense that the condition number varies proportionately with the size of the subproblem, that is the ratio $\frac{H}{h}$. The fact that the condition number is independent of the jumps, is also evident from the table. 

Following our analysis, there is a minimum number of eigenfunctions (corresponding to the bad eigenvalues) which should be added in the enrichment for the method to be robust with respect to the contrast. In order to see that, in our next experiment, we choose one particular discretization ($H$, $h$) and distribution of $\alpha$, and run our algorithm each time with a fixed number of eigenfunctions for all subdomains. For the experiment we have chosen $H=1/6$ (i.e. $6\times 6$ subdomains) and $h=1/36$, and the results are presented in Table \ref{tbl:2}. As we can see from the table, the condition number improves as more and more eigenfunctions are included in the enrichment, but stops (or improves very slowly) once the sixth eigenfunction has been included. So the minimum number of eigenfunctions in this case is six. This also agrees with the adaptive version, cf. Table \ref{tbl:1}, where the same test case needed six eigenfunctions.      

In our final experiment, we compare the two algorithms, corresponding to $\scalebox{0.8}{TYPE}=\scalebox{0.8}{LAYER}$ and $\scalebox{0.8}{TYPE}=\scalebox{0.8}{SUBD}$. Although they are very similar in their convergence behavior (their condition number estimates are very close to each other), the former requires far less number of eigenvalues in order to achieve the same level of convergence whenever there are inclusions or channels inside a subdomain, cf. Fig. \ref{fig:bargraph}, suggesting that the $\scalebox{0.8}{TYPE}=\scalebox{0.8}{LAYER}$ algorithm has a clear advantage over the other in such cases.  

\begin{figure}[htb]
\centerline{
\hspace{-.8cm}\includegraphics[width=1\textwidth,height=.5\textwidth]{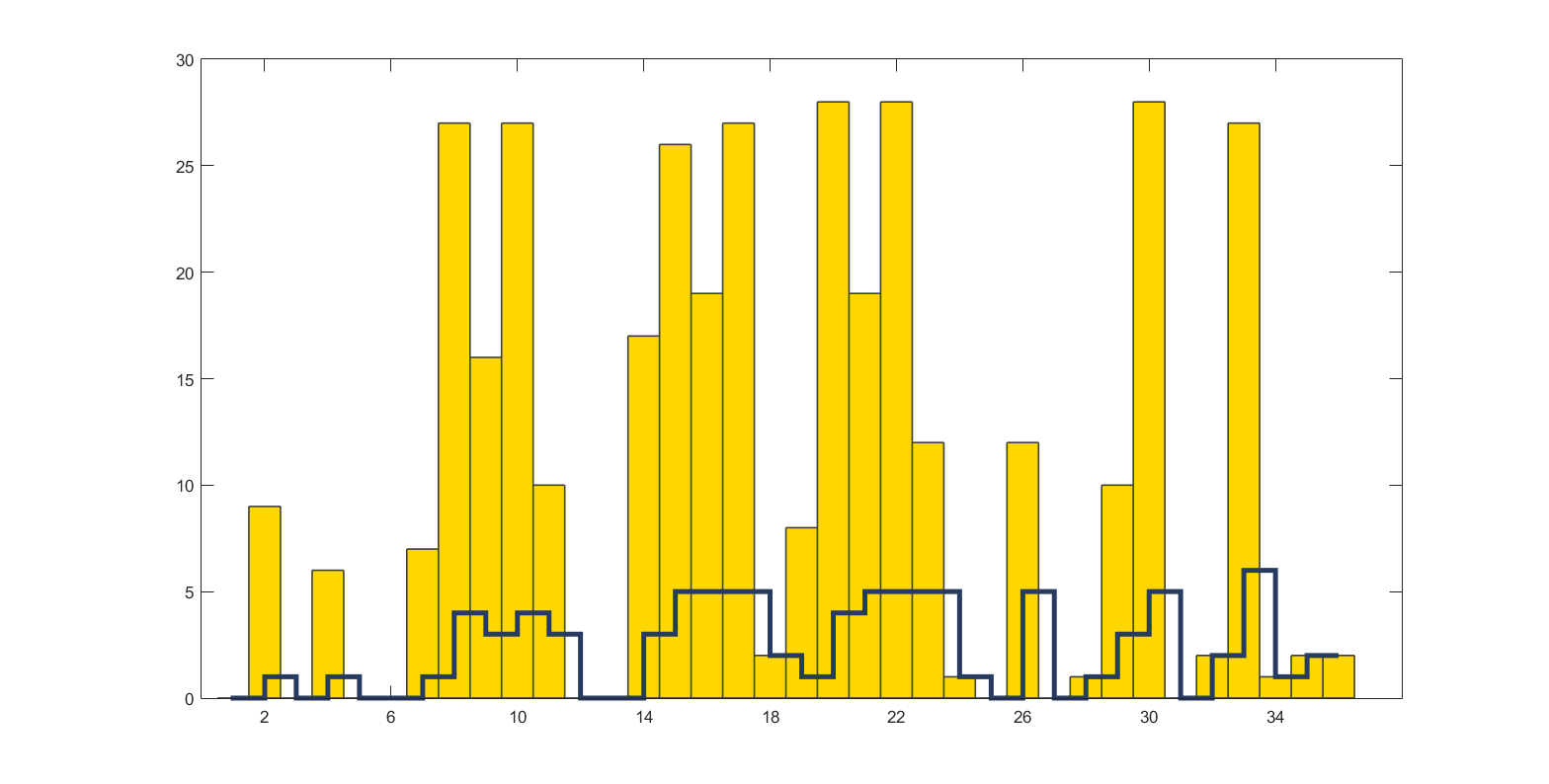}}
\caption{Showing the number of eigenvalues in each subdomain ($36$ subdomains), with values larger than the 100, selected for the enrichment. Here $H=1/6$ (i.e. $6\times 6$ subdomains) and $h=1/36$, $\alpha_b=1$, $\alpha_c=1e4$, and $\alpha_i=1e6$. The bar graph corresponds to {\footnotesize $TYPE = SUBD$} algorithm, 
and the stair graph corresponds to {\footnotesize $TYPE = LAYER$ algorithm}.
}
\label{fig:bargraph}
\end{figure}

\section*{Acknowledgments}
Leszek Marcinkowski was partially
supported by Polish Scientific Grant 2016/21/B/ST1/00350.

\bibliographystyle{siam}
\bibliography{aschwarz} 

\end{document}